\documentclass[12pt]{amsart}%
\usepackage{amsmath,amsthm,amsfonts,amssymb}
\usepackage{verbatim}
\usepackage{latexsym}
\usepackage{marginnote}
\usepackage{amscd}
\usepackage{amssymb}
\usepackage{xcolor}
\usepackage{amsmath}
\usepackage{amsfonts}
\usepackage{mathrsfs}
\usepackage{graphicx}%
\providecommand{\U}[1]{\protect\rule{.1in}{.1in}}
\providecommand{\U}[1]{\protect\rule{.1in}{.1in}}
\providecommand{\U}[1]{\protect\rule{.1in}{.1in}}
\newtheorem{theorem}{Theorem}[section]
\newtheorem{lemma}[theorem]{Lemma}

\begin{document}
\title{Sobolev spaces of vector-valued functions on  compact groups}
\author{Yaogan Mensah}

\address{Department of Mathematics, University of Lom\'e, Togo}
\email{\textcolor[rgb]{0.00,0.00,0.84}{mensahyaogan2@gmail.com, ymensah@univ-lome.org}}

\begin{abstract}
This paper deals with a class of Sobolev spaces of vector-valued functions on a compact group. Using some results among which are  the inversion formula and the Plancherel type theorem involving the Fourier transform of vector-valued functions, we define Sobolev spaces of Bessel potential type.  Then, some Sobolev embedding  theorems are proved. 
\end{abstract}
\maketitle


\section{Introduction}
Sobolev spaces have proven their effectiveness in mathematical sciences. They are well studied on  certain classic spaces such as euclidean spaces and on more general differential manifolds. Some of  them can be constructed via the Fourier transform and therefore they can be  studied with techniques from abstract/classical harmonic analysis.

Some  studies of Sobolev spaces in the context of abstract harmonic analysis can be found in  \cite{Gorka1, Gorka2, Krukowski, Mensah2, Kumar}. Particularly in \cite{Kumar}, the authors introduced Sobolev spaces of complex-valued functions over  compact groups and study their properties. They obtained, among other results, some Sobolev embedding and compactness theorems. 
 
 The aim of this paper is to study the vector valued  aspect  of some results in \cite{Kumar}. More precisely, we introduce  Sobolev spaces of vector-valued functions on a compact group and proved some Sobolev embedding  theorems.
 
The rest of the paper is organized as follows. Section \ref{Preliminaries} is devoted to preliminaries on harmonic analysis of vector-valued functions on compact groups.  Section \ref{main results} contains our main results.

\section{Preliminaries}\label{Preliminaries}
In this section, we very briefly recall some facts concerning harmonic analysis on compact groups, mainly their representation theory \cite{Hewitt} and the Fourier transform of vector-valued functions defined on such groups \cite{Assiamoua}.

Let $G$ be a compact Hausdorff group which may not be necessary abelian. A concrete  exemple of  such a  group is the special unitary  group  $SU(2)$  consisting of matrices $A=\left(\begin{array}{cc}
a & b \\ 
-\overline{b} & \overline{a}
\end{array} \right)$
where $a,b\in \mathbb{C}$ are such that $|a|^2+|b|^2=1$.  A unitary representation $\sigma$ of $G$  on a Hilbert space $H_\sigma$ is a  homomorphism  $\sigma : G\longrightarrow \mathcal{U}(H_\sigma)$ where  $\mathcal{U}(H_\sigma)$ denotes the group of unitary operators on $H_\sigma$. The Hilbert space $H_\sigma$ is called the representation space of $\sigma$ and the dimension of $H_\sigma$  is called the dimension of $\sigma$ and it is denoted $d_\sigma$. A  unitary representation $\sigma$ is said to be continuous if the mapping $G\longrightarrow H_\sigma, x\longmapsto \sigma (x)\xi$ is continuous for every $\xi\in H_\sigma$. 
 A  representation $\sigma$ of $G$ on  $H_\sigma$ is called irreducible if there is no proper closed subspace $M$ of $H_\sigma$ which is  invariant by $\sigma$, that is, $\forall x\in G, \forall \xi \in M, \sigma (x)\xi\in M$. 
It is well known that  the dimension of any  unitary  irreducible  representation of a compact group  is  of finite dimension \cite{Hewitt}. 
Two unitary representions $\sigma_i,\, i=1,2$ of $G$ on $H_{\sigma_i},\,i=1,2$ are said to be unitary equivalent if there exists a unitary linear operator $T: H_{\sigma_1}\longrightarrow H_{\sigma_2}$
such that $\forall x\in G, \,\sigma_2(x)T=T\sigma_1(x)$.

Let us denote by $\widehat{G}$  the the set of equivalent classes  of unitary irreducible  representations of $G$. It is called the unitary  dual of $G$ and it is discrete since $G$ is compact. 

For $\sigma \in \widehat{G}$, choose an orthonormal basis $\lbrace \xi_1^\sigma,\cdots,\xi_{d_\sigma}^\sigma\rbrace$ of $H_\sigma$. The coefficients of the representation $\sigma$ are the functions $u_{i,j}^\sigma$ defined by 
$$u_{i,j}^\sigma (x)=\langle \sigma (x)\xi_i^\sigma,\xi_j^\sigma \rangle_\sigma,\, x\in G,$$
where $\langle \cdot, \cdot\rangle_\sigma$ is the inner product of the Hilbert space $H_\sigma$.

Let $E$ be a complex  Banach space. Denote by $L^1(G, E)$ the set $E$-valued Bochner integrable functions on $G$. 

Let $f\in L^1(G, E)$. Following \cite{Assiamoua}, the Fourier transform   $\widehat{f}$ of $f$ is the collection $\left(\widehat{f}(\sigma)\right)_{\sigma \in \widehat{G}}$  of sesquilinear maps where for each $\sigma$, the sesquilinear map $\widehat{f}(\sigma)$ is defined from $H_\sigma\times H_\sigma$ into $E$,    by 
\begin{equation}\label{Fourier}
\widehat{f}(\sigma)(\xi,\eta)=\int_G\langle \sigma (x)^*\xi,\eta\rangle_\sigma f(x)dx,\,\xi,\eta \in  H_\sigma.
\end{equation}
Since $G$ is compact, the  space of $E$-valued Bochner square integrable functions on $G$, denoted by   $L^2(G, E)$, is a subspace of $L^1(G, E)$. Clearly, the Fourier transform of functions in $L^2(G,E)$ is well-defined by formula (\ref{Fourier}). 
For $f\in L^2(G,E)$, the inversion formula is given by

\begin{equation}
f(x)=\sum\limits_{\sigma\in \widehat{G}}d_\sigma \sum\limits_{i=1}^{d_\sigma}\sum\limits_{j=1}^{d_\sigma}\widehat{f}(\sigma)(\xi_j^\sigma, \xi_i^\sigma)u_{i,j}^\sigma (x)
\end{equation}

Denote by $\mathscr{S}(H_\sigma\times H_\sigma, E)$ the set of $E$-valued sesquilinear maps on $H_\sigma\times H_\sigma$. 
Set $$\mathscr{S}(\widehat{G},E)=\prod\limits_{\sigma\in \widehat{G}}\mathscr{S}(H_\sigma\times H_\sigma, E).$$

Define $\mathscr{S}_p(\widehat{G}, E), p\geqslant 1$, to be the set of elements $\phi$ of  $\mathscr{S}(\widehat{G}, E)$ such that 
 $$\sum\limits_{\sigma\in \widehat{G}}d_\sigma\sum\limits_{i=1}^{d_\sigma}\sum\limits_{j=1}^{d_\sigma}\|\phi(\xi_j^\sigma,\xi_i^\sigma)\|_E^p<\infty.$$
 Also,  consider on $\mathscr{S}_p(\widehat{G}, E)$ the norm 
 \begin{equation}
\|\phi\|_{\mathscr{S}_p}=\left(\sum\limits_{\sigma\in \widehat{G}}d_\sigma\sum\limits_{i=1}^{d_\sigma}\sum\limits_{j=1}^{d_\sigma}\|\phi(\xi_j^\sigma,\xi_i^\sigma)\|_E^p\right)^{\frac{1}{p}}.
\end{equation}
The proof of completeness  and other properties of the spaces $\mathscr{S}_p(\widehat{G}, E)$ can be found in \cite{Mensah1}. 

The space  $\mathscr{S}_2(\widehat{G},E)$ is of particular interest: the Fourier transformation is an isometry from $L^2(G,E)$ onto $\mathscr{S}_2(\widehat{G},E)$ \cite{Assiamoua}. 

Finally, let us recall the following well known fact which we will use later. For $x=(x_1,\cdots,x_n)$ where $x_1,\cdots,x_n$ are real (or complex) numbers  and for $p\geqslant 1$, set 
$$\|x\|_p=\left(|x_1|^p+\cdots +|x_n|^p\right)^{\frac{1}{p}}.$$
It is known that if $1\leqslant p\leqslant q\leqslant\infty$, then
\begin{equation}\label{Comparison}
\|x\|_q \leqslant \|x\|_p \leqslant n^{\frac{1}{p}-\frac{1}{q}}\|x\|_q.
\end{equation}

\section{Main results}\label{main results}
In this section, we introduce the Sobolev spaces of $E$-valued functions on the compact Hausdorff group $G$ and prove some embedding results. Throughout the rest of the paper, the symbol $X\hookrightarrow Y$ means that the space $X$ is continuously embedded in the space $Y$.

 Let $\left(\gamma (\sigma)\right)_{\sigma\in\widehat{G}}$ be a sequence of nonnegative real numbers. Pick  $s$ in $[0,\infty)$.  The Sobolev space $H_\gamma^s(G,E)$ is defined to be the subspace of  $L^2(G,E)$ consisting in functions $f$ such  that 
 $$\sum\limits_{\sigma\in \widehat{G}}d_\sigma(1+\gamma (\sigma)^2)^s\sum\limits_{i=1}^{d_\sigma}\sum\limits_{j=1}^{d_\sigma}\|\widehat{f}(\sigma)(\xi_j^\sigma,\xi_i^\sigma)\|_E^2<\infty.$$
 The following norm is defined on $H_\gamma^s(G,E)$ :
\begin{equation}
\|f\|_{H_\gamma^s}=\left(\sum\limits_{\sigma\in \widehat{G}}d_\sigma(1+\gamma (\sigma)^2)^s\sum\limits_{i=1}^{d_\sigma}\sum\limits_{j=1}^{d_\sigma}\|\widehat{f}(\sigma)(\xi_j^\sigma,\xi_i^\sigma)\|_E^2\right)^{\frac{1}{2}}.
\end{equation}
 
\begin{theorem}
The space $H_\gamma^s(G,E)$ is a Banach space.
\end{theorem}
\begin{proof}
The map $f\longmapsto \left(1+\gamma (\sigma)^2\right)^{\frac{s}{2}}\widehat{f}$ is an isometric bijection from $H_{\gamma}^s (G,E))$ onto $\mathscr{S}_2(\widehat{G},E)$. Since $\mathscr{S}_2(\widehat{G},E)$ is a Banach space, then so is $H_{\gamma}^s (G,E))$.
\end{proof}

\begin{theorem}
If $t>s$,  then $H_\gamma^t(G,E)\hookrightarrow H_\gamma^s(G,E)$ with $\|f\|_{H_\gamma^s}\leqslant \|f\|_{H_\gamma^t}$.
\end{theorem}
\begin{proof}
The result comes from the fact that if $t>s$, then $(1+\gamma (\sigma)^2)^t>(1+\gamma (\sigma)^2)^s$ since $1+\gamma (\sigma)^2>1$.
\end{proof}
\begin{theorem} We  have
$H_\gamma^s(G,E)\hookrightarrow L^2(G,E)$ with $\|f\|_{L^2}\leqslant \|f\|_{H_\gamma^s}$.
\end{theorem}
\begin{proof}
Let $f\in H_\gamma^s(G,E)$. Then, 
\begin{align*}
\|f\|_{L^2}^2&=\|\widehat{f}\|_{\mathscr{S}_2}^2\\
 & =\sum\limits_{d_\sigma\in \widehat{G}}d_\sigma\sum\limits_{i=1}^{d_\sigma}\sum\limits_{j=1}^{d_\sigma}\|\widehat{f}(\sigma)(\xi_j^\sigma,\xi_i^\sigma)\|_E^2\\
 &\leqslant\sum\limits_{d_\sigma\in \widehat{G}}d_\sigma(1+\gamma (\sigma)^2)^s\sum\limits_{i=1}^{d_\sigma}\sum\limits_{j=1}^{d_\sigma}\|\widehat{f}(\sigma)(\xi_j^\sigma,\xi_i^\sigma)\|_E^2\\
 &=\|f\|_{H_\gamma^s}^2.
\end{align*}
\end{proof}

\begin{lemma}\label{Strong}
Let $\sigma$ be a continuous representation of $G$. Let $a\in G$ and let $\varepsilon>0$. There exists a neighborhood $U$ of $a$
such that $$\forall x\in U,\, |u_{i,j}^\sigma(x)-u_{i,j}^\sigma(a)|<\varepsilon.$$
\end{lemma}
\begin{proof}
 Since the representation $\sigma$ is continuous, there exists  a neighbourhood $U$ of $a$ such that
 $\|\sigma(x)-\sigma(a)\|<\varepsilon$ whenever $x\in U$.
\begin{align*}
|u^\sigma_{i,j}(x)-u^\sigma_{i,j}(a)|&=|\langle \sigma(x)\xi^\sigma_i,\xi^\sigma_j\rangle_\sigma-\langle \sigma(a)\xi^\sigma_i,\xi^\sigma_j\rangle_\sigma|\\
 &=|\langle \left(\sigma(x)-\sigma(a)\right)\xi^\sigma_i, \xi^\sigma_j\rangle_\sigma|\\
 &\leqslant \|\sigma(x)-\sigma(a)\|\|\xi^\sigma_i\|\|\xi^\sigma_j\|\\
 &\leqslant \|\sigma(x)-\sigma(a)\|<\varepsilon.
 \end{align*}
 where meanwhile we have used the Cauchy-Schwarz inequality, the boundedness of the operator $\sigma(x)-\sigma(a)$ and the fact that the $\xi_i^\sigma$'s are unit vectors.
\end{proof}
\begin{lemma}\label{bound}
Assume that $\sum\limits_{\sigma\in \widehat{G}}\displaystyle\frac{d_\sigma ^3}{(1+\gamma(\sigma)^2)^s}<\infty$. If  $f\in H_\gamma^s(G,E)$,  then $f$ is continuous on $G$.
\end{lemma}
\begin{proof}
Let $a\in G$ and let $U$ be like in Lemma \ref{Strong}. 
Let $f\in H_\gamma^s(G,E)$. If $x\in U$, then
\begin{align*}
\|f(x)-f(a)\|_E  &=\left\|\sum\limits_{\sigma\in \widehat{G}}d_\sigma\sum\limits_{i=1}^{d_\sigma}\sum\limits_{j=1}^{d_\sigma}\widehat{f}(\sigma)(\xi_j^\sigma,\xi_i^\sigma)(u_{i,j}^\sigma(x)-u_{i,j}^\sigma(a)) \right\|_E\\
&\leqslant \sum\limits_{\sigma\in \widehat{G}}d_\sigma\sum\limits_{i=1}^{d_\sigma}\sum\limits_{j=1}^{d_\sigma}\left\|\widehat{f}(\sigma)(\xi_j^\sigma,\xi_i^\sigma)\right\|_{E}|u_{i,j}^\sigma(x)-u_{i,j}^\sigma(a)| \\
& \leqslant\varepsilon \sum\limits_{\sigma\in \widehat{G}}d_\sigma\sum\limits_{i=1}^{d_\sigma}\sum\limits_{j=1}^{d_\sigma}\left\|\widehat{f}(\sigma)(\xi_j^\sigma,\xi_i^\sigma)\right\|_{E}\,\mbox{(by Lemma \ref{Strong})}\\
&=\varepsilon \sum\limits_{\sigma\in \widehat{G}}d_\sigma\sum\limits_{i=1}^{d_\sigma}\sum\limits_{j=1}^{d_\sigma}(1+\gamma(\sigma)^2)^{\frac{s}{2}}\left\|\widehat{f}(\sigma)(\xi_j^\sigma,\xi_i^\sigma)\right\|_{E}(1+\gamma(\sigma)^2)^{-\frac{s}{2}}.
\end{align*}
Now, applying the H\"older inequality, we have
\begin{align*}
\|f(x)-f(a)\|_E &\leqslant\varepsilon \left(\sum\limits_{\sigma\in \widehat{G}}d_\sigma\sum\limits_{i=1}^{d_\sigma}\sum\limits_{j=1}^{d_\sigma} (1+\gamma(\sigma)^2)^s\|\widehat{f}(\sigma)(\xi_j^\sigma,\xi_i^\sigma)\|^2_{E}\right)^{\frac{1}{2}}\times \\
&\times\left(\sum\limits_{\sigma\in \widehat{G}}d_\sigma\sum\limits_{i=1}^{d_\sigma}\sum\limits_{j=1}^{d_\sigma}      (1+\gamma (\sigma)^2)^{-s}\right)^{\frac{1}{2}}\\
&=\varepsilon \left(\sum\limits_{\sigma\in \widehat{G}}d_\sigma\sum\limits_{i=1}^{d_\sigma}\sum\limits_{j=1}^{d_\sigma} (1+\gamma(\sigma)^2)^s\|\widehat{f}(\sigma)(\xi_j^\sigma,\xi_i^\sigma)\|^2_{E}\right)^{\frac{1}{2}}\left( \sum\limits_{\sigma\in \widehat{G}}d_\sigma^3(1+\gamma (\sigma)^2)^{-s}\right)^{\frac{1}{2}}\\
&=\varepsilon \|f\|_{H_\gamma^s}\left(       \sum\limits_{\sigma\in \widehat{G}}d_\sigma^3(1+\gamma (\sigma)^2)^{-s}\right)^{\frac{1}{2}}.
\end{align*}
Thus, $f$ is continuous at $a$. Since $a$ is an arbitrary element of $G$, then $f$ is continuous on $G$.
\end{proof}

\begin{lemma}\label{bound}
Assume that $\sum\limits_{\sigma\in \widehat{G}}\displaystyle\frac{d_\sigma ^3}{(1+\gamma(\sigma)^2)^s}<\infty$. If $f\in  H_\gamma^s(G,E)$, then there exists  a constant $C(\gamma,s)$, depending only on $\gamma$ and $s$, such that $$\|f\|_{\infty}:=\sup\{\|f(x)\|_E : x\in G\}\leqslant C(\gamma,s)\|f\|_{H_\gamma^s}.$$ 
\end{lemma}
\begin{proof}
 Let $x\in G$. Then,
\begin{align*}
\|f(x)\|_{E}&=\left\|\sum\limits_{\sigma\in \widehat{G}}d_\sigma\sum\limits_{i=1}^{d_\sigma}\sum\limits_{j=1}^{d_\sigma}\widehat{f}(\sigma)(\xi_j^\sigma,\xi_i^\sigma)u_{i,j}^\sigma(x)\right\|_{E}\\
&\leqslant \sum\limits_{\sigma\in \widehat{G}}d_\sigma \sum\limits_{i=1}^{d_\sigma}\sum\limits_{j=1}^{d_\sigma}|u_{i,j}^\sigma(x)|\| \widehat{f}(\sigma)(\xi_j^\sigma,\xi_i^\sigma)\|_{E}.
\end{align*}
Since $\sigma$ is a unitary representation, then by the Cauchy-Schwarz inequality, $$|u^\sigma_{i,j}(x)|=|\langle \sigma(x)\xi_i,\xi_j\rangle|\leqslant \|\sigma (x)\|\|\xi_i^\sigma\|\|\xi_j^\sigma\|=1.$$
Then, 
\begin{align*}
\|f(x)\|_{E}&\leqslant \sum\limits_{\sigma\in \widehat{G}}d_\sigma \sum\limits_{i=1}^{d_\sigma}\sum\limits_{j=1}^{d_\sigma}\| \widehat{f}(\sigma)(\xi_j^\sigma,\xi_i^\sigma)\|_{E}\\
&=\sum\limits_{\sigma\in \widehat{G}}d_\sigma\sum\limits_{i=1}^{d_\sigma}\sum\limits_{j=1}^{d_\sigma}(1+\gamma(\sigma)^2)^{\frac{s}{2}}\left\|\widehat{f}(\sigma)(\xi_j^\sigma,\xi_i^\sigma)\right\|_{E}(1+\gamma(\sigma)^2)^{-\frac{s}{2}}.
\end{align*}
By the H\"older inequality, we obtain
\begin{align*}
\|f(x)\|_{E}&\leqslant \left(\sum\limits_{\sigma\in \widehat{G}}d_\sigma\sum\limits_{i=1}^{d_\sigma}\sum\limits_{j=1}^{d_\sigma} (1+\gamma(\sigma)^2)^s\|\widehat{f}(\sigma)(\xi_j^\sigma,\xi_i^\sigma)\|^2_{E}\right)^{\frac{1}{2}}\left(\sum\limits_{\sigma\in \widehat{G}}d_\sigma\sum\limits_{i=1}^{d_\sigma}\sum\limits_{j=1}^{d_\sigma}      (1+\gamma (\sigma)^2)^{-s}\right)^{\frac{1}{2}}\\
&= \left(\sum\limits_{\sigma\in \widehat{G}}d_\sigma\sum\limits_{i=1}^{d_\sigma}\sum\limits_{j=1}^{d_\sigma} (1+\gamma(\sigma)^2)^s\|\widehat{f}(\sigma)(\xi_j^\sigma,\xi_i^\sigma)\|^2_{E}\right)^{\frac{1}{2}}\left(       \sum\limits_{\sigma\in \widehat{G}}d_\sigma^3(1+\gamma (\sigma)^2)^{-s}\right)^{\frac{1}{2}}\\
&= \|f\|_{H_\gamma^s}\left(\sum\limits_{\sigma\in \widehat{G}}d_\sigma^3(1+\gamma (\sigma)^2)^{-s}\right)^{\frac{1}{2}}\\
&=C(\gamma,s)\|f\|_{H_\gamma^s}<\infty
\end{align*}
where $C(\gamma,s)=\left(\sum\limits_{\sigma\in \widehat{G}}d_\sigma^3(1+\gamma (\sigma)^2)^{-s}\right)^{\frac{1}{2}}$. Hence,  $$\|f\|_\infty\leqslant C(\gamma,s)\|f\|_{H_\gamma^s}.$$
\end{proof}
Let us denote by $\mathcal{C}(G,E)$ the space of $E$-valued continuous  functions on $G$.
\begin{theorem}
If $\sum\limits_{\sigma\in\widehat{G}}\displaystyle\frac{d_\sigma ^3}{(1+\gamma(\sigma)^2)^s}<\infty$, then $H_\gamma^s(G,E)\hookrightarrow \mathcal{C}(G,E)$. 
\end{theorem}
\begin{proof}
This theorem is the conjunction of Lemma \ref{Strong}  by which $f$ is continuous  and Lemma \ref{bound} by which  the continuous embedding inequality $\|f\|_\infty\leqslant C(\gamma,s)\|f\|_{H_\gamma^s}$ holds.
\end{proof}

\begin{lemma}\label{douzoa}
Let $\phi\in \prod\limits_{\sigma \in \widehat{G}}\mathscr{S}(H_\sigma \times H_\sigma, E)$. If $1\leqslant p\leqslant q$, then 
$$\left(\sum\limits_{i=1}^{d_\sigma}\sum\limits_{j=1}^{d_\sigma}\|\phi(\sigma)(\xi_j^\sigma,\xi_i^\sigma)\|_E^p\right)^{\frac{1}{p}}\leq (d_\sigma ^2)^{\frac{1}{p}-\frac{1}{q}}\left(\sum\limits_{i=1}^{d_\sigma}\sum\limits_{j=1}^{d_\sigma}\|\phi(\sigma)(\xi_j^\sigma,\xi_i^\sigma)\|_E^q\right)^{\frac{1}{q}}.$$
\end{lemma}
\begin{proof}
Use the right handside inequality in (\ref{Comparison}).
\end{proof}
\begin{theorem}
Let $t>s>0$. Set $\alpha'=\displaystyle\frac{2t}{t-s}$. If  $\sum\limits_{\sigma\in \widehat{G}}\displaystyle\frac{d_\sigma^3}{(1+\gamma (\sigma)^2)^t}<\infty$, 

then  $H_\gamma^s(G,E)\hookrightarrow L^{\alpha'}(G,E)$ with 
$\|f\|_{L^{\alpha'}}\leqslant\left(\sum\limits_{\sigma\in \widehat{G}}\frac{d_\sigma^3}{(1+\gamma (\sigma)^2)^t} \right)^{\frac{s}{2t}} \|f\|_{H_\gamma^s}.$
\end{theorem}
\begin{proof}
Let $\alpha$ be the H\"older conjugate of $\alpha'$. That is $\displaystyle\frac{1}{\alpha}+ \frac{1}{\alpha'}=1$. Then, $\alpha=\displaystyle\frac{2t}{s+t}$ and $\displaystyle\frac{s}{t}=\frac{2-\alpha}{\alpha}$. It follows that $1<\alpha<2$. By the inverse Hausdorff-Young inequality \cite[Lemma 5.1]{Garcia}, we have 
$$\|f\|_{L^{\alpha'}}\leqslant\|\widehat{f}
\|_{\mathscr{S}_\alpha}.$$

We have $\|\widehat{f}
\|_{\mathscr{S}_\alpha}^\alpha =\sum\limits_{\sigma\in\widehat{G}}d_\sigma\sum\limits_{i=1}^{d_\sigma}\sum\limits_{j=1}^{d_\sigma}\|\widehat{f}(\sigma)(\xi_j^\sigma, \xi_i^\sigma)\|_{E}^\alpha$. Using Lemma \ref{douzoa} with the fact that $1<\alpha<2$, we  attain
$$\left(\sum\limits_{i=1}^{d_\sigma}\sum\limits_{j=1}^{d_\sigma}\|\widehat{f}(\sigma)(\xi_j^\sigma, \xi_i^\sigma)\|_{E}^\alpha\right)^{\frac{1}{\alpha}}\leqslant (d_\sigma^2)^{\frac{1}{\alpha}-\frac{1}{2}}\left(\sum\limits_{i=1}^{d_\sigma}\sum\limits_{j=1}^{d_\sigma}\|\widehat{f}(\sigma)(\xi_j^\sigma, \xi_i^\sigma)\|_{E}^2\right)^{\frac{1}{2}}.$$
The latter inequality implies 
$$\sum\limits_{i=1}^{d_\sigma}\sum\limits_{j=1}^{d_\sigma}\|\widehat{f}(\sigma)(\xi_j^\sigma, \xi_i^\sigma)\|_{E}^\alpha\leqslant d_\sigma^{2-\alpha}\left(\sum\limits_{i=1}^{d_\sigma}\sum\limits_{j=1}^{d_\sigma}\|\widehat{f}(\sigma)(\xi_j^\sigma, \xi_i^\sigma)\|_{E}^2\right)^{\frac{\alpha}{2}}.$$
Therefore, 
\begin{align*}
\|\widehat{f}
\|_{\mathscr{S}_\alpha}^\alpha & \leqslant \sum\limits_{\sigma\in \widehat{G}}d_\sigma d_\sigma^{2-\alpha}\left(\sum\limits_{i=1}^{d_\sigma}\sum\limits_{j=1}^{d_\sigma}\|\widehat{f}(\sigma)(\xi_j^\sigma, \xi_i^\sigma)\|_{E}^2\right)^{\frac{\alpha}{2}}\\
&=\sum\limits_{\sigma\in \widehat{G}}d_\sigma (1+\gamma (\sigma)^2)^{\frac{s\alpha}{2}} \left(\sum\limits_{i=1}^{d_\sigma}\sum\limits_{j=1}^{d_\sigma}\|\widehat{f}(\sigma)(\xi_j^\sigma, \xi_i^\sigma)\|_{E}^2\right)^{\frac{\alpha}{2}}d_\sigma^{2-\alpha}(1+\gamma (\sigma)^2)^{-\frac{s\alpha}{2}}.
\end{align*}

Observe that $\displaystyle\frac{1}{\frac{2}{\alpha}}+\displaystyle\frac{1}{\frac{2}{2-\alpha}}=1$.
Now, apply  the H\"older inequality to obtain 
\begin{align*}
\|\widehat{f}
\|_{\mathscr{S}_\alpha}^\alpha & \leqslant \left(\sum\limits_{\sigma\in \widehat{G}}d_\sigma (1+\gamma (\sigma)^2)^s\sum\limits_{i=1}^{d_\sigma}\sum\limits_{j=1}^{d_\sigma}\|\widehat{f}(\sigma)(\xi_j^\sigma, \xi_i^\sigma)\|_{E}^2 \right)^{\frac{\alpha}{2}}\left(\sum\limits_{\sigma\in \widehat{G}}d_\sigma d_\sigma ^2 (1+\gamma (\sigma)^2)^{-\frac{s\alpha}{2-\alpha}}   \right)^{\frac{2-\alpha}{2}}\\
&=\left(\sum\limits_{\sigma\in \widehat{G}}d_\sigma (1+\gamma (\sigma)^2)^s\sum\limits_{i=1}^{d_\sigma}\sum\limits_{j=1}^{d_\sigma}\|\widehat{f}(\sigma)(\xi_j^\sigma, \xi_i^\sigma)\|_{E}^2 \right)^{\frac{\alpha}{2}}\left(\sum\limits_{\sigma\in \widehat{G}} \frac{d_\sigma^3}{(1+\gamma (\sigma)^2)^{\frac{s\alpha}{2-\alpha}}  }   \right)^{\frac{2-\alpha}{2}}.
\end{align*}
Thus,  
\begin{align*}
\|\widehat{f}
\|_{\mathscr{S}_\alpha} &\leqslant \left(\sum\limits_{\sigma\in \widehat{G}}d_\sigma (1+\gamma (\sigma)^2)^s\sum\limits_{i=1}^{d_\sigma}\sum\limits_{j=1}^{d_\sigma}\|\widehat{f}(\sigma)(\xi_j^\sigma, \xi_i^\sigma)\|_{E}^2 \right)^{\frac{1}{2}}\left(\sum\limits_{\sigma\in \widehat{G}} \frac{d_\sigma^3}{(1+\gamma (\sigma)^2)^{\frac{s\alpha}{2-\alpha}} } \right)^{\frac{2-\alpha}{2\alpha}}\\
&=\|f\|_{H_\gamma^s}\left(\sum\limits_{\sigma\in \widehat{G}}\frac{d_\sigma^3}{(1+\gamma (\sigma)^2)^t} \right)^{\frac{s}{2t}}.
\end{align*}
Finally, $\|f\|_{L^{\alpha'}}\leqslant\|\widehat{f}
\|_{\mathscr{S}_\alpha}\leqslant \|f\|_{H_\gamma^s}\left(\sum\limits_{\sigma\in \widehat{G}}\frac{d_\sigma^3}{(1+\gamma (\sigma)^2)^t} \right)^{\frac{s}{2t}}.$
\end{proof}

\section{Conclusion}
This paper, explores the embeddings of Sobolev-type spaces consisting of vector-valued functions on  compact groups. The main results include
continuous embeddings between Sobolev spaces,  between
Sobolev spaces and spaces of continuous functions and  between Sobolev and Lebesgue spaces.


\end{document}